\documentclass[reqno,10pt]{amsart}

\usepackage{amsmath,amsfonts,amssymb,amsthm,epsfig,amscd,}
\usepackage[]{fontenc}
\usepackage[latin1]{inputenc}
\usepackage{enumerate}
\usepackage[cmtip,all]{xy}
\usepackage{tabularx,multirow}
\usepackage{color}
\usepackage{enumitem}
\setlist[itemize]{itemsep=0ex,parsep=0ex,topsep=0ex,partopsep=0ex,leftmargin=4.5ex}

 \allowdisplaybreaks \numberwithin{equation}{section}

\newtheorem{theorem}{Theorem}[section]
\newtheorem{step}{Step}
\newtheorem{result}{Result}

\newtheorem{claim}[theorem]{Claim}

\theoremstyle{definition}

\theoremstyle{remark}

\DeclareMathOperator{\codim}{codim}

\def\leq{{\leqslant}}
\def\geq{{\geqslant}}

\begin{document}

\title[On complete intersections containing a linear subspace]{On complete intersections containing a linear subspace}

\thanks{ This collaboration has benefitted of funding from the
research project \emph{``Families of curves: their moduli and their related varieties"} (CUP: E81-18000100005) - Mission Sustainability - University of Rome Tor Vergata.}

\author{Francesco Bastianelli}
\address{Francesco Bastianelli, Dipartimento  di Matematica, Universit\`{a} degli Studi di Bari ``Aldo Moro", Via Edoardo Orabona 4, 70125 Bari -- Italy}
\email{francesco.bastianelli@uniba.it}

\author{Ciro Ciliberto}
\address{Ciro Ciliberto, Dipartimento  di Matematica, Universit\`{a} degli Studi di Roma ``Tor Vergata", Viale della Ricerca Scientifica 1, 00133 Roma -- Italy}
\email{cilibert@mat.uniroma2.it}

\author{Flaminio Flamini}
\address{Flaminio Flamini, Dipartimento  di Matematica, Universit\`{a} degli Studi di Roma ``Tor Vergata", Viale della Ricerca Scientifica 1, 00133 Roma -- Italy}
\email{flamini@mat.uniroma2.it}

\author{Paola Supino}
\address{Paola Supino, Dipartimento  di Matematica e Fisica, Universit\`{a} degli Studi ``Roma Tre", Largo S. L. Murialdo 1, 00146 Roma -- Italy}
\email{supino@mat.uniroma3.it}

\begin{abstract} Consider the Fano scheme $F_k(Y)$ parameterizing $k$--dimensional linear  subspaces contained in a  complete intersection $Y \subset \mathbb{P}^m$ of multi--degree $\underline{d} = (d_1, \ldots, d_s)$.
It is known that, if  $t := \sum_{i=1}^s \binom{d_i +k}{k}-(k+1) (m-k)\leqslant 0$  and $\Pi_{i=1}^sd_i >2$, for $Y$ a general complete intersection as above, then $F_k(Y)$ has dimension $-t$.
In this paper we consider the case $t> 0$.
Then the locus $W_{\underline{d},k}$ of all complete intersections as above containing a $k$--dimensional linear subspace is irreducible and turns out to have codimension $t$ in the parameter space of all complete intersections with the given multi--degree. Moreover, we prove that for general $[Y]\in W_{\underline{d},k}$ the scheme $F_k(Y)$ is zero--dimensional of length one. This implies that $W_{\underline{d},k}$ is rational.
\end{abstract}

\maketitle

\section{Introduction}

In this paper we will be concerned with the {\em Fano scheme} $F_k(Y)$, parameterizing $k$--dimensional linear subspaces contained in a subvariety $Y \subset \mathbb{P}^m$, when $Y$ is a complete intersection of multi--degree $\underline{d} = (d_1, \ldots, d_s)$, with $1\leqslant s \leqslant m-2$.
We will assume that $Y$ is neither a linear subspace nor a quadric, cases to be considered as trivial.
Thus we will constantly assume that  $\Pi_{i=1}^s\;d_i >2$.

Let $S_{\underline{d}}:= \bigoplus_{i=1}^s H^{0}\left(\mathbb P^ m,\mathcal{O}_{\mathbb{P}^{m}}(d_i) \right)$, and consider its Zariski open subset $S^*_{\underline{d}}:= \bigoplus_{i=1}^s \big ( H^{0}\left(\mathbb P^ m,\mathcal{O}_{\mathbb{P}^{m}}(d_i) \right)\setminus \{0\}\big )$.
For any $u := (g_1,\ldots,g_s) \in S^*_{\underline{d}}$, let $Y_u:=V (g_1,\ldots,g_s)\subset \mathbb{P}^m$ denote  the closed subscheme defined by the vanishing of the polynomials $g_1,\ldots,g_s$. When $u \in S^*_{\underline{d}}$ is general, $Y_u$ is a smooth, irreducible variety of dimension $m-s \geqslant 2$.
For any integer $k \geq 1$, we define the locus
$$
W_{\underline{d},k}:=\left\{\left.u\in S^*_{\underline{d}} \right| \; F_k(Y_u) \neq \emptyset \; \right\} \subseteq S^*_{\underline{d}}$$and set
\[
t(m,k,\underline {d}):=\sum_{i=1}^s \binom{d_i +k}{k} - (k+1)(m-k).
\]
If no confusion arises, we will simply denote $t(m,k,\underline {d})$ by $t$.

First of all, consider the case $t \leq 0$. This is the most studied case in the literature, and it is now well understood (cf.\;e.g.\;\cite{Bo,DeMa, Morin,Pre}). In particular, the following holds.

\begin{result}\label{res1} Let $m,k,s$ and $\underline{d}=(d_1,\dots,d_s)$ be such that $\Pi_{i=1}^s d_i >2$ and $t\leq 0$. Then:
\begin{itemize}
\item[(a)] $W_{\underline{d},k} = S_{\underline{d}}^*$;
\item[(b)] for general $u \in S_{\underline{d}}^*$,  $F_k(Y_u)$ is smooth, of  dimension $\dim(F_k(Y_u)) = -t$ and it is
irreducible when  $\dim(F_k(Y_u))\geq 1$.
\end{itemize}
\end{result}
The proof of this result can be found e.g. in \cite[Prop.2.1,\;Cor.2.2,\;Thm.\;4.1]{Bo}, for the complex case, and in \cite[Thm.\;2.1,\;(b)\;\&\;(c)]{DeMa}, for any algebraically closed field. In addition, in \cite[Thm.\;4.3]{DeMa} the authors compute $\deg(F_k(Y_u))$ under the Pl\"ucker embedding $F_k(Y_u) \subset \mathbb{G}(k,m) \hookrightarrow \mathbb{P}^N$, with $N={{m+1}\choose {k+1}}-1$. Their formulas extend  to any $k \geq 1$ enumerative formulas by Libgober in \cite{L}, who computed $\deg(F_1(Y_u))$ when $t(m,1,\underline {d})=0$.

On the other hand, we are interested in the case $t >  0$, where the known results can be summarized as follows.

\begin{result}\label{res2}  Let $m,k,s$ and $\underline{d}=(d_1,\dots,d_s)$ be such that $\Pi_{i=1}^s d_i >2$ and $t> 0$. Then:
\begin{itemize}
\item[(a)] $W_{\underline{d},k} \subsetneq S_{\underline{d}}^*$.
\item[(b)] $W_{\underline{d},k} $ contains points $u$ for which $Y_u \subset \mathbb{P}^m$ is a smooth complete intersection of dimension $m-s$  if and only if $s \leq m-2k$.
\item[(c)] For $s\leq m-2k$, set $H_{\underline{d},k} := \left\{ \left.u \in W_{\underline{d},k} \right| Y_u \subset \mathbb{P}^m \text{\emph{ is smooth, of dimension }}m-s\right\}$.
If $d_i \geq 2$ for any $1 \leq i \leq s$, then $H_{\underline{d},k}$ is irreducible, unirational and
$\codim_{S^*_{\underline{d}}} (H_{\underline{d},k} ) = t$. Moreover,
for general $u \in H_{\underline{d},k}$, $F_k(Y_u)$ is a zero--dimensional scheme.
\end{itemize}
\end{result}

The proof of Result \ref{res2}\;(a) is contained in \cite[Thm.\;2.1\;(a)]{DeMa}, whereas that of assertions (b) and (c) is contained in \cite[Cor.\;1.2,\;Rem.\;3.4]{Miyazaki}; both proofs therein hold for any algebraically closed field.

The main result of this paper, which improves on Result \ref{res2},  is the following.

\begin{theorem}\label{thm:Predonzan+}  Let $m,k,s$ and $\underline{d}=(d_1,\dots,d_s)$ be such that $\Pi_{i=1}^s d_i >2$ and $t> 0$.  Then $W_{\underline{d},k} \subsetneq S^*_{\underline{d}}$ is non--empty, irreducible and rational, with $\codim_{S^*_{\underline{d}}} (W_{\underline{d},k}) = t$.
Furthermore, for a general point $u \in W_{\underline{d},k}$, the variety $Y_u \subset \mathbb{P}^m$ is a complete intersection of dimension $m-s$ whose Fano scheme $F_k(Y_u)$ is a zero--dimensional scheme of length one.
Moreover, $Y_u$ has singular locus of dimension $\max\{-1,2k+s-m-1\}$ along its unique $k$--dimensional linear subspace (in particular $Y_u$ is smooth if and only if $m-s\geq 2k$).
\end{theorem}

The proof of this theorem  is  contained in Section \ref{S2} and it extends \cite[Prop.\;2.3]{BCFS} to arbitrary $k\geq 1$. Theorem \ref{thm:Predonzan+} improves, via different and easier methods, Miyazaki's results in \cite[Cor.\;1.2]{Miyazaki}, showing that for general $u \in W_{\underline{d},k}$ one has $\deg(F_k(Y_u))=1$, which implies the rationality of $W_{\underline{d},k}$. Moreover we also get rid of Miyazaki's hypothesis $m-s\geq 2k$.

\section{The proof}\label{S2} This section is  devoted to the proof of Theorem \ref{thm:Predonzan+}.

\begin{proof}[Proof of Theorem \ref{thm:Predonzan+}] Let $\mathbb{G}:= \mathbb{G}(k,m)$ be the Grassmannian  of $k$-linear subspaces in $\mathbb P^m$ and consider the incidence correspondence
$$
J:=\left\{\left.\left(\left[\Pi\right], u \right) \in \mathbb{G} \times S^*_{\underline{d}}\right| \Pi  \subset Y_u\right\} \subset \mathbb{G} \times S^*_{\underline{d}}
$$
with the two  projections
$$\mathbb{G} \stackrel{\pi_1}{\longleftarrow} J \stackrel{\pi_2}{\longrightarrow} S^*_{\underline{d}}.$$
The map $\pi_1\colon J \to \mathbb{G}$ is surjective and, for any $[\Pi]\in \mathbb{G}$, one has
$\pi_1^{-1}\left(\left[\Pi\right]\right) =\bigoplus_{i=1}^s \big ( H^{0}\left(\mathcal{I}_{\Pi/\mathbb{P}^{m}}(d_i)\right) \setminus \{0 \}\big ),$ where $\mathcal{I}_{\Pi/\mathbb{P}^{m}}$ denotes the ideal sheaf of $\Pi$ in $\mathbb{P}^{m}$.

Thus $J$ is irreducible with
$\dim (J) = \dim (\mathbb{G}) + \dim (\pi_1^{-1}\left(\left[\Pi\right]\right)) =  (k+1)(m-k) + \sum_{i=1}^s h^0\left(\mathcal{I}_{\Pi/\mathbb{P}^{m}}(d_i)\right)$. From the exact sequence
\begin{equation}\label{eq:exact}
0\to \bigoplus_{i=1}^s  \mathcal{I}_{\Pi/\mathbb{P}^{m}}(d_i) \to \bigoplus_{i=1}^s  \mathcal{O}_{\mathbb{P}^{m}}(d_i) \to \bigoplus_{i=1}^s   \mathcal{O}_{\Pi}(d_i) \cong \bigoplus_{i=1}^s   \mathcal{O}_{\mathbb P^k}(d_i) \to 0,
\end{equation}
one gets
\begin{equation}\label{eq:dimJ}
\dim (J) = (k+1)(m-k) + \sum_{i=1}^s \binom{d_i +m}{m}  -\sum_{i=1}^s \binom{d_i +k}{k}= \dim (S_{\underline{d}}^*) - t.
\end{equation}

The next step recovers \cite[Cor.\;1.2]{Miyazaki}  via different and easier methods, and we also get rid of the hypothesis $m-s\geq 2k$ present there. We essentially adapt the argument in \cite[Proof of Prop.\,2.1]{Bo},
used for the case $t \leq 0$.

\begin{step}\label{step1}   The map $\pi_2 \colon J \to S^*_{\underline{d}}$ is generically finite onto its image
 $W_{\underline{d},k}$, which is therefore irreducible and unirational. Moreover $\codim_{S^*_{\underline{d}}}(W_{\underline{d},k}) = t$.\\
For general $u \in W_{\underline{d},k}$,  $F_k(Y_u)$ is a zero--dimensional scheme and
$Y_u$ has singular locus of dimension $\max\{-1,2k+s-m-1\}$ along any of the $k$--dimensional linear subspaces in  $F_k(Y_u)$.
\end{step}
\begin{proof}[Proof of Step \ref{step1}] One has $W_{\underline{d},k}=\pi_2(J)$, hence $W_{\underline{d},k} \subsetneq S_{\underline{d}}^*$ is irreducible and unirational, because $J$ is  rational, being an open dense subset of a vector bundle over $\mathbb{G}$.
Once one shows that $\pi_2\colon J \to W_{\underline{d},k}$ is generically finite, one deduces that $\codim_{S^*_{\underline{d}}}(W_{\underline{d},k}) = t$ from \eqref{eq:dimJ}.
 Therefore, we focus on proving that $\pi_2$ is generically finite, i.e. that
if $u\in W_{\underline{d},k}$ is a general point, then $\dim (\pi_2^{-1}(u))=0$.

Let $[\Pi] \in \mathbb{G}$ and choose $[y_0,y_1, \ldots, y_m]$ homogeneous coordinates in $\mathbb{P}^m$ such that the  ideal of $\Pi$ is $I_{\Pi} := (y_{k+1},\dots,y_m)$. For general
$
\left(\left[\Pi\right],u\right) \in \pi_1^{-1}\left(\left[\Pi\right]\right)\subset J$, with $u = (g_1, \ldots, g_s) \in W_{\underline{d},k}$,
we can write
$$
g_i = \sum_{h=k+1}^m y_h \; p_i^{(h)} + r_i, \; 1 \leq i \leq s,
$$ with
\begin{equation}\label{eq:Pt}
r_i \in (I^2_{\Pi})_{d_i} \;\; {\rm whereas} \;\; p_i^{(h)} = \sum_{|\underline{\mu}|= d_i-1} c^{(h)}_{i,\underline{\mu}} \;\; \underline{y}^{\underline{\mu}} \in \mathbb{C}[y_0,y_1, \ldots, y_k]_{d_i-1}, \quad 1 \leq i \leq s, \; k+1 \leq h \leq m,
\end{equation} where $(I^2_{\Pi})_{d_i}$ is the homogenous component of degree $d_i$ of the ideal $I^2_\Pi$,  $\underline{\mu} := (\mu_0, \cdots, \mu_k) \in \mathbb{Z}^{k+1}_{\geqslant 0}$, $|\underline{\mu}|:= \sum_{r=0}^k \mu_r$, and $\underline{y}^{\underline{\mu}} := y_0^{\mu_0} y_1^{\mu_1} \cdots y_k^{\mu_k}$. By the generality assumption on $u$, the polynomials $p_i^ {(h)}$ and $r_i$ are general.

The Jacobian matrix $(\frac {\partial g_i}{\partial y_j})_{1\leq	 i\leq s;0\leq j\leq m}$ computed along $\Pi$ takes the block form
\[
M= ( {\bf 0} \quad  {\bf P} ) \quad \text{where}\quad {\bf P}:=(p_i^ {(h)})_{1\leq i\leq s;k+1 \leq h \leq m}
\]
where the $\bf 0$--block has size $s\times (k+1)$ and ${\bf P}$ has size $s\times (m-k)$, where $m-k\geq s$ because of course $\dim(Y_u)=m-s\geq k$. By the generality of the polynomials $p_i^ {(h)}$, the locus of $\Pi$ where ${\rm rk}(M)<s$, which coincides with the singular locus of $Y_u$ along $\Pi$,  has dimension $\max\{-1,2k+s-m-1\}$ and, by Bertini's theorem, it coincides with the singular locus of $Y_u$.

Next we consider the following exact sequence of normal sheaves
\begin{equation}\label{eq:normalbundle}
0 \to N_{\Pi/Y_u} \to N_{\Pi/\mathbb{P}^m} \cong \mathcal{O}_{\mathbb{P}^{k}}(1)^{\oplus (m-k)} \to \left. N_{Y_u/\mathbb{P}^m}\right|_{\Pi} \cong \bigoplus_{i=1}^s\mathcal{O}_{\mathbb{P}^{k}}(d_i)
\end{equation}
(see \cite[Lemma 68.5.6]{stack}).
Any $\xi \in H^0(\Pi, N_{\Pi/\mathbb{P}^m})$ can be identified with a collection of $m-k$ linear forms on $\Pi \cong \mathbb{P}^k$
\[
\varphi_h^{\xi}(\underline{y}) := a_{h,0} y_0 + a_{h,1} y_1 + \cdots + a_{h,k} y_k, \; k+1 \leq h \leq m,
\]
 whose coefficients fill up the $(m-k) \times (k+1)$ matrix
$$
A_\xi:= (a_{h,j}),\; k+1 \leq h \leq m, \; 0 \leq j \leq k;
$$by abusing notation, one may identify  $\xi$ with  $A_{\xi}$.

Thus the map $H^0\left(\Pi, N_{\Pi/\mathbb{P}^m}\right) \stackrel{\sigma}{\longrightarrow} H^0\left(\Pi, \left.N_{Y_{u}/\mathbb{P}^m}\right|_{\Pi}\right)$,
arising from  \eqref{eq:normalbundle}  is given by (cf. e.g. \cite[formula (4)]{Bo})
\begin{equation}\label{eq:linsys}
A_{\xi} \stackrel{\sigma}{\longrightarrow} \left(\sum_{0 \leq j \leq k < h \leq m} a_{h,j} y_j p_i^{(h)} \right)_{1 \leq i \leq s}.
\end{equation} Notice that the assumption $t > 0$ reads as
$$(k+1) (m-k) = h^0\left(\Pi, N_{\Pi/\mathbb{P}^m}\right) < h^0\left(\Pi, \left.N_{Y_{u}/\mathbb{P}^m}\right|_{\Pi}\right)=\sum_{i=1}^s \binom{d_i+k}{k}.$$

\begin{claim}\label{claim:2} The map
$H^0\left(\Pi, N_{\Pi/\mathbb{P}^m}\right) \stackrel{\sigma}{\longrightarrow} H^0\left(\Pi, \left.N_{Y_{u}/\mathbb{P}^m}\right|_{\Pi}\right)$ is injective, equivalently $h^0(N_{\Pi/{Y_{u}}}) =0$.
In particular, for a general point $u \in W_{\underline{d},k}$, the Fano scheme $F_k({Y_u})$ contains $\{[\Pi]\}$ as a zero--dimensional integral component.
\end{claim}
\begin{proof}[Proof of Claim \ref{claim:2}] Using \eqref{eq:Pt}, the polynomials  on the right--hand--side of \eqref{eq:linsys} read as
$$\sum_{h=k+1}^m \; \sum_{j=0}^k \; a_{h,j} y_j \left( \sum_{|\underline{\mu}|= d_i-1} c^{(h)}_{i,\underline{\mu}} \;\; \underline{y}^{\underline{\mu}} \right) ,\;\; 1 \leq i \leq s.$$
Ordering the previous polynomial expressions via the standard lexicographical monomial order on the canonical basis $\{\underline{y}^{\underline{\mu}}\}$ of
$\mathbb{C}[y_0,y_1, \ldots, y_k]_{d_i} = H^0(\mathcal{O}_{\mathbb P^k}(d_i)), \;\; 1 \leq i \leq s$, the injectivity of the map $\sigma$ is equivalent for the homogeneous linear system
\begin{equation}\label{eq:linsysBorcea}
\sum_{0 \leq j \leq k < h \leq m} \; c^{(h)}_{i, \underline{\nu} - \underline{e}_j} \; a_{h,j} = 0, \; 1 \leq i \leq s,
\end{equation}
to have  only the trivial solution, where $\underline{\nu} := (\nu_0, \nu_1, \ldots, \nu_{k}) \in \mathbb{Z}^{k+1}_{\geqslant 0}$ is such that $|\underline{\nu}| = d_i$, $\underline{e}_j$ is the $(j+1)$--th vertex of the standard $(k+1)$-simplex in $\mathbb{Z}^{k+1}_{\geqslant 0} \setminus \{\underline{0}\}$, and $c^{(h)}_{i, \underline{\nu} - \underline{e}_j} = 0$ when $\underline{\nu} - \underline{e}_j \notin \mathbb{Z}^{k+1}_{\geqslant 0}$
(this last condition stands for ``$\underline{\nu} - \underline{e}_j$ {\em improper}" as formulated in \cite[p.\; 29]{Bo}). The linear system \eqref{eq:linsysBorcea} consists of
$\sum_{i=1}^s \binom{d_i+k}{k}$ equations in the $(k+1)(m-k)$ indeterminates $a_{h,j}$, with coefficients $c^{(h)}_{i,\underline{\mu}}$, $0 \leq j \leq k < h \leq m$.

Let $C := (c^{(h)}_{i, \underline{\nu} - \underline{e}_j})$ be the coefficient matrix of \eqref{eq:linsysBorcea}; one is reduced to show that, for general choices of the entries
$c^{(h)}_{i, \underline{\nu} - \underline{e}_j}$, the matrix $C$ has maximal rank $(k+1)(m-k)$. This can be done arguing as in \cite[p.\;29]{Bo}.
Namely, row--indices of $C$ are determined by the standard lexicographical monomial order on the canonical basis of $\bigoplus_{i=1}^s \mathbb{C} [y_0, y_1, \ldots, y_k]_{d_i}$, whereas
column--indices of $C$ are determined by the standard lexicographic order on the set of indices $(h,j)$.
If one considers the square sub--matrix $\widehat{C}$ of $C$
formed by the first $(k+1)(m-k)$ rows and by all the columns of $C$,  then $\det(\widehat{C})$ is a non--zero polynomial
in the indeterminates $c^{(h)}_{i, \underline{\mu}}$. Indeed, take the lexicographic order on the set of indices
$$(h, \; i, \; \underline{\mu}), \;\; \mbox{where} \;\; k+1 \leq h \leq m, \; |\underline{\mu}| = d_i-1, \; 1 \leq i \leq s,$$
and order the monomials appearing in the expression of $\det(\widehat{C})$ according to the following rule: the monomials $m_1$ and $m_2$ are such that $m_1 > m_2$ if, considering the smallest index $(h, \; i, \; \underline{\mu})$ for which $c^{(h)}_{i, \underline{\mu}}$ occurs in the monomial $m_1$ with exponent $p_1$ and in the monomial $m_2$ with exponent $p_2 \neq p_1$, one has $p_1 > p_2$.
The greatest monomial (in the monomial ordering described above) appearing in $\det(\widehat{C})$ has coefficient $\pm \; 1$, since in each column the choice of the $c^{(h)}_{i,  \underline{\mu}}$ entering in this monomial is uniquely determined. By maximality of such monomial, it follows that
$\det(\widehat{C}) \neq 0$, which shows that $C$ has maximal rank $(k+1)(m-k)$, i.e. the map $\sigma$ is injective.

The injectivity of $\sigma$ and \eqref{eq:normalbundle} yield $h^0(N_{\Pi/{Y_{u}}}) =0$.  Since $H^0(N_{\Pi/{Y_{u}}})$ is the tangent space to $F_k(Y_u)$ at its point $[\Pi]$, one deduces that $\{[\Pi]\}$ is a zero--dimensional, reduced component  of
$F_k(Y_u)$, as claimed.
\end{proof}

Finally, by monodromy arguments, the irreducibility of $J$ and Claim \ref{claim:2} ensure that for general $u \in W_{\underline{d},k}$, the Fano scheme $F_k(Y_u)$ is zero-dimensional and reduced, i.e. $\pi_2\colon J \to W_{\underline{d},k}$ is generically finite, and that $Y_u$ has a singular locus of dimension $\max\{-1,2k+s-m-1\}$ along any of the $k$--dimensional linear subspaces in  $F_k(Y_u)$.  This  completes the proof of
Step \ref{step1}. \end{proof}

To conclude the proof of Theorem \ref{thm:Predonzan+}, we need the following numerical result.

\begin{step}\label{lemma:sistemino} For $0\leq h\leq k-1$ integers, consider the integer
\begin{equation*}\label{eq:sistemino}
\delta_h(m,k, \underline{d}):=\sum_{i=1}^s\binom{d_i  +k}{k}-\sum_{i=1}^s\binom{d_i+h}{h}- (k-h)(m+h+1-k).
\end{equation*}
If $\delta_h(m,k, \underline{d}) \leqslant 0$, then
$$t (m,k, \underline{d})\leqslant 0.$$
\end{step}

\begin{proof} [Proof of Step \ref{lemma:sistemino}] In order to ease notation, we set $\delta_h:=\delta_h(m,k, \underline{d})$. Therefore, the condition $\delta_h \leqslant 0$ implies $m \geq \frac{1}{k-h}
 \left[
 \sum_{i=1}^s \binom{d_i +k}{k} - \binom{d_i +h}{h}\right] -(h+1-k)$. Plugging the previous inequality in the expression of $t$, one has
\begin{equation}\label{eq:calcolot}
t \leq - \sum_{i=1}^s\left[\frac{h+1}{k-h}\binom{d_i+k}{k}
 - \frac{k+1}{k-h}\binom{d_i+h}{h}\right] +(k+1)(h+1).
\end{equation} Set $D(x):= \frac{h+1}{k-h}\binom{x+k}{k}-\frac{k+1}{k-h} \binom{x+h}{h}.$ Thus, \eqref{eq:calcolot} reads
\begin{equation}\label{eq:calclot2}
t  \leq - \sum_{i=1}^s D(d_i)+(k+1)(h+1).
\end{equation}The assumption $0\leq h\leq k-1$ gives
 \begin{equation*}\label{eq:Dd}
\begin{split}
D(d_i)&=\frac{(h+1)(d_i+1)\cdots(d_i+h)}{k!(k-h)}
 \bigg((d_i+h+1)\cdots(d_i+k)-(k+1)k\cdots(h+2)
\bigg), \; 1 \leq i \leq s.
\end{split}
\end{equation*}
The polynomial $D(x)$  vanishes for $x=1$, which is its only positive root.
Notice that
\begin{equation*}
D(2)= \frac{h+1}{k-h}\binom{k+2}{k}-\frac{k+1}{k-h} \binom{h+2}{h}=\frac{(h+1)(k+1)}{2}>0.
\end{equation*} In particular, $D(x)$ is increasing and positive for $x>1$, so from \eqref{eq:calclot2} it follows that
\begin{equation*}
t \leq -\sum_{i=1}^s D(d_i)+(k+1)(h+1) \leq - s \; D(2) + (k+1)(h+1) =(k+1)(h+1)\bigg(1-\frac{s}{2} \bigg).
\end{equation*}
Therefore, when $s\geq 2$, we have $t\leq 0$ and we are done in this case.

If  $s=1$, set $d:= d_1$. In this case \eqref{eq:calclot2} is $t \leq -D(d)+(k+1)(h+1)$, where again $D(d)$ is increasing and positive for $d>1$. When $s=1$, we have $d \geq 3$ by assumption. Thus, one computes
{\small
\begin{equation*}
D(3)=(k+1)(h+1) \frac{k+h+5}{6}
\end{equation*}}and so, for any $d \geq 3$,  one has
\begin{equation*}
t  \leq  -D(d)+(k+1)(h+1) \leq -D(3)+(k+1)(h+1) = (k+1)(h+1) \frac{1-k-h}{6}.
\end{equation*}
Being $0 \leq h \leq k-1$, one deduces that $t\leq 0$, completing the proof of Step \ref{lemma:sistemino}.
\end{proof}

The final step of the proof of Theorem \ref{thm:Predonzan+} is the following.

\begin{step}\label{step3} For general $u \in W_{\underline{d},k}$, the zero--dimensional Fano scheme $F_k(Y_u)$ has length one.
In particular, the map $\pi_2 \colon J \to W_{\underline{d},k}$ is birational and $W_{\underline{d},k}$ is rational.
\end{step}

\begin{proof}[Proof of Step \ref{step3}] Let us consider the (locally closed) incidence correspondence
$$
I:=\left\{\left.\left(\left[\Pi_1\right],\left[\Pi_2\right], u \right) \in \mathbb{G} \times\mathbb{G} \times S^*_{\underline{d}}\right| \Pi_1  \neq\Pi_2,\; \Pi_i  \subset Y_u, \; 1 \leq i \leq 2\right\} \subset \mathbb{G} \times\mathbb{G} \times S^*_{\underline{d}}.
$$
If $I$  is not empty, let $\varphi\colon I \to J$ be the map defined by $$\varphi\left(\left(\left[\Pi_1\right],\left[\Pi_2\right], u \right)\right)=\left(\left[\Pi_1\right], u \right).$$
We need to prove that $\varphi$  is not dominant. To do this, consider the (locally closed) subset
$$
I_h:=\left\{\left.\left(\left[\Pi_1\right],\left[\Pi_2\right], u \right) \in I \; \right| \; \Pi_1\cap\Pi_2  \cong \mathbb{P}^h\right\}, \; {\rm where} \;  -1\leq h\leq k-1
$$
(we set $\mathbb P^ {-1}=\emptyset$, i.e. the case $h=-1$ occurs when $\Pi_1$ and $\Pi_2$ are skew). Clearly, one has
$I=\bigsqcup_{h =-1}^{k-1} I_h$. Setting $\varphi_h := \varphi_{|I_h}$, it is sufficient to prove that $\varphi_h$ is not dominant, for any $-1\leq h\leq k-1$.

So, let $h$ be such that $I_h$ is not empty, and let $T_h$ be an irreducible component of $I_h$.
Of course, if $\dim (T_h) < \dim(J)$, the restriction ${\varphi_{h|T_h}}\colon T_h\to J$ is not dominant.
On the other hand, suppose that $\dim (T_h) > \dim (J)$.
For any such a component, the map ${\varphi_{h|T_h}}$ cannot be dominant, otherwise the composition $T_h \stackrel{{\varphi_{h |T_h}}}{\longrightarrow} J \stackrel{\pi_2}{\longrightarrow} W_{\underline{d},k}$ would be dominant, as $\pi_2$ is, which would imply that the general fiber of $\pi_2$ is positive dimensional, contradicting Step \ref{step1}.

Therefore, it remains to investigate the case $\dim (T_h) = \dim  (J)$.
We  estimate the dimension of $T_h$ as follows. Consider
$$\mathbb{G}^2_h:=
 \left\{\left(\left[\Pi_1\right],\left[\Pi_2\right]) \in \mathbb{G}\times \mathbb{G} \right | \Pi_1\cap\Pi_2  \cong \mathbb{P}^h\right\}\subset \mathbb{G}\times \mathbb{G},
$$
which is locally closed in $\mathbb{G}\times \mathbb{G}$.
The projection
$$\widehat{\pi}_1\colon \mathbb{G}^2_h \to \mathbb{G}, \;\; \left(\left[\Pi_1\right],\left[\Pi_2\right]\right) {\longmapsto} \left[\Pi_1\right]$$
is surjective onto $\mathbb{G}$ and any $\widehat{\pi}_1$--fiber is irreducible, of dimension equal to $\dim \left(\mathbb{G}(h,k)\times \mathbb{G}(k-h-1,m-h-1) \right)= (h+1)(k-h)+(k-h)(m-k)$. Thus
$$
\dim\mathbb{G}^2_h= (k+1)(m-k) +(h+1)(k-h)+(k-h)(m-k).
$$
One has the projection
$$\psi_h\colon T_h {\longrightarrow} \mathbb{G}^2_h, \;\;\; \left(\left[\Pi_1\right],\left[\Pi_2\right], u\right) {\longmapsto} \left(\left[\Pi_1\right], \left[\Pi_1\right]\right),$$
which is surjective, because the projective group acts transitively on $\mathbb{G}^2_h$.
Hence $\dim (T_h) = \dim  (\mathbb{G}^2_h) + \dim  (\mathfrak{F_h})$, where
$\mathfrak{F_h} := \bigoplus_{i=1}^s \Big ( H^{0}\left(\mathcal{I}_{\Pi_1\cup\Pi_2/\mathbb{P}^{m}}(d_i)\right)\setminus \{0\}\Big )$ is the general fiber of ${\psi_{h\vert T_h}}$ and
where $\mathcal{I}_{\Pi_1\cup\Pi_2/\mathbb{P}^{m}}$ denotes the ideal sheaf of $\Pi_1\cup\Pi_2$ in $\mathbb{P}^{m}$.

\begin{claim}\label{cl:paf} For every positive integer $d$ one has
\[ h^ 0(\mathcal{I}_{\Pi_1\cup\Pi_2/\mathbb{P}^{m}}(d))=\dim(S_d)- 2 \binom{d +k}{k} + \binom{d +h}{h}. \]
  \end{claim}

  \begin{proof}[Proof of Claim \ref  {cl:paf}] We have
  \begin{equation}\label{eq:acc}
  h^ 0(\mathcal{I}_{\Pi_1/\mathbb{P}^{m}}(d))=\dim(S_d)-  \binom{d +k}{k}.
 \end{equation}
Consider the linear system $\Sigma$ cut out on $\Pi_2$ by $|\mathcal{I}_{\Pi_1/\mathbb{P}^{m}}(d)|$. We claim that $\Sigma$ is the complete linear system of hypersurfaces of degree $d$ of $\Pi_2$ containing $\Pi:=\Pi_1\cap \Pi_2$. Indeed $\Sigma$  contains all hypersurfaces consisting of a hyperplane through $\Pi$ plus a hypersurface of degree $d-1$ of $\Pi_2$, which proves our claim.
In the light of this fact, and arguing as in \eqref{eq:exact} and \eqref{eq:dimJ}, we deduce that
\begin{equation*}
h^ 0(\mathcal{I}_{\Pi_1\cup \Pi_2/\mathbb{P}^{m}}(d))= h^ 0(\mathcal{I}_{\Pi_1/\mathbb{P}^{m}}(d))-\left(\dim(\Sigma)+1\right)=h^0(\mathcal{I}_{\Pi_1/\mathbb{P}^{m}}(d))-\Bigg (\binom{d +k}{k}-\binom{d +h}{h}\Bigg),
\end{equation*}
which, by \eqref {eq:acc}, yields the assertion. \end{proof}

By Claim \ref {cl:paf} we have
\begin{equation*}
\begin{split}
\dim  (\mathfrak{F_h})&=\dim  (S_{\underline{d}}^*) - 2 \sum_{i=1}^s \binom{d_i +k}{k}
  + \sum_{i=1}^s \binom{d_i +h}{h}.
 \end{split}
\end{equation*}Hence
\begin{equation}\label{eq:*14Nov}
\begin{split}
  \dim  (T_h) & = \dim  (\mathfrak{F_h}) + \dim  (\mathbb{G}^2_h)=\\
  &=\dim  (S_{\underline{d}}^*) - 2 \sum_{i=1}^s \binom{d_i +k}{k}
  + \sum_{i=1}^s \binom{d_i +h}{h}+\\
  &+
  (k+1)(m-k) +(h+1)(k-h)+(k-h)(m-k)=\\
  &= \dim  (J) - \sum_{i=1}^s \binom{d_i +k}{k}
  + \sum_{i=1}^s \binom{d_i +h}{h}+(k-h)(m+h+1-k) = \\
	&= \dim  (J) - \delta_h.
 \end{split}
\end{equation} Since $\dim  (T_h) = \dim  (J)$, \eqref{eq:*14Nov} implies $\delta_h =0$.
When $0 \leq h\leq k-1$, Step \ref{lemma:sistemino} gives
$t \leq 0$, contrary to our assumption.  When $h=-1$, one has
$0=\delta_{-1} = t$, again against our assumptions.

Since no component $T_h\subset I_h$ can dominate $J$, the map $\varphi\colon I\to J$ is not dominant.
We conclude therefore that the map $\pi_2 \colon J \to W_{\underline{d},k} $ is birational, completing the proof of Step \ref{step3}.
\end{proof}

Steps \ref{step1}--\ref{step3} prove Theorem \ref{thm:Predonzan+}.
\end{proof}


\section*{Acknowledgements}
We would like to thank Enrico Fatighenti and Francesco Russo for helpful discussions.


\end{document}